\documentclass[a4paper,10pt,leqno]{amsart}
\usepackage{amsmath,amsfonts,amsthm,amssymb,dsfont}

\usepackage[T1]{fontenc}
\usepackage{enumerate}
\usepackage{mathrsfs}
\usepackage{mathtools}
\usepackage{enumerate}
\usepackage{caption}
\usepackage{graphicx}
\usepackage{overpic}

\usepackage[dvipsnames]{xcolor}
\usepackage[alphabetic]{amsrefs}
\usepackage{float}

\newtheorem{thm}{Theorem}[section]
\newtheorem{lem}[thm]{Lemma}
\newtheorem{lem-defin}[thm]{Lemma-Definition}
\newtheorem{prop}[thm]{Proposition}
\newtheorem{cor}[thm]{Corollary}
\newtheorem{theoremalph}{Theorem}

\newtheorem*{corC}{Corollary C}
\newtheorem*{corB}{Corollary B}
\newtheorem*{Question}{Question}

\theoremstyle{definition}
\newtheorem{rem}[thm]{Remark}

\newtheorem{defin}[thm]{Definition}

\newtheorem{notation}[thm]{Notation}

\newcommand{\Z}{\mathbb{Z}}

\makeatletter
\def\paragraph{\@startsection{paragraph}{4}%
  \z@\z@{-\fontdimen2\font}%
  {\normalfont\bfseries}}
\makeatother

\begin{document}

\title[Characterising large-type Artin groups]{Characterising large-type Artin groups}

\author{Alexandre Martin and Nicolas Vaskou}

\maketitle

\begin{abstract}
We show that the class of large-type Artin groups is invariant under isomorphism, in stark contrast with the corresponding situation for Coxeter groups. We obtain this result by providing a purely algebraic  
characterisation of large-type Artin groups (i.e. independent of the presentation graph). 

As a corollary, we completely describe the Artin groups isomorphic to a given large-type Artin group, and characterise those large-type Artin groups that are rigid.
\end{abstract}

\section{Introduction}

Artin groups and Coxeter groups are two actively studied families of groups that are defined by means of a presentation graph, i.e. a finite simplicial graph with edges  labelled by integers greater than~$1$. While the structure and geometry of Coxeter groups is by now well understood, many basic open problems remain unsolved for general Artin groups (see \cite{C-ProblemsArtin} for a survey). As a result, several classes of Artin groups have been introduced, where the additional conditions allow for new tools that can be used to study these groups. These additional conditions generally involve either restrictions on the presentation graph (as with right-angled, irreducible, large-type Artin groups), or properties of the associated Coxeter group (as with spherical-type, hyperbolic-type, affine-type Artin groups). One thing to emphasise however is that these classes all involve a choice of a presentation graph, and it could a priori be possible for an Artin group of a given class to be isomorphic to an Artin group outwith that class. It is thus natural to ask whether any of the classes involve properties of the group itself rather than a specific presentation:

\begin{Question}
Which classes of Artin groups are invariant under isomorphism?
\end{Question}

 This article focuses on the class of \textbf{large-type} Artin groups, that is, where all the labels of the presentation graph are at least $3$. 
 It is already known that the class of large-type Coxeter groups is not invariant under isomorphism. Indeed, the dihedral group $D_{6}$ of order $12$ (with presentation graph an edge labelled $6$) is isomorphic to the direct product $D_3\times \Z_2$ (with presentation graph a triangle with labels $(2, 2, 3)$). By contrast, our main result is that, for Artin groups, being of large type is indeed a property of the group, and not of the chosen presentation graph:

\begin{theoremalph}
The class of large-type Artin groups is invariant under isomorphism.	
 \end{theoremalph}

 \medskip

Understanding whether classes of Artin groups are invariant under isomorphism, beside being of interest in and of itself, has wider implications for the Isomorphism Problem for Artin groups,  that is, the problem of determining which presentation graphs yield isomorphic Artin groups. As with many open problems about Artin groups, this problem is currently wide open for general Artin groups. (The corresponding Isomorphism Problem for Coxeter groups is also open, but much more is currently known, see for instance \cite{M-IsomorphismCoxeterSurvey}.) Instead most of the existing literature focuses on solving the Isomorphism Problem \textit{within} a given class of Artin groups. For instance, Droms solved the Isomorphism Problem within the class of right-angled Artin groups \cite{D-IsomorphismRAAGs}, and Paris solved the Isomorphism Problem within the class of spherical-type Artin groups \cite{P-IsomorphismSpherical}. More recently, the second author solved the Isomorphism Problem within the class of large-type Artin groups \cite{V-RigidityLarge}.   A possible approach to  describing the Artin groups isomorphic to a given  Artin group is to first show that it belongs to a manageable class of Artin groups that is invariant under isomorphism, and to then solve the Isomorphism Problem within that class. This approach can be carried out for large-type Artin groups:   Theorem~A together with the second author's solution to the Isomorphism Problem within that class allows us to completely describe the Artin groups isomorphic to a given large-type Artin group. In particular, we can characterise those large-type Artin groups that are \textit{rigid}, i.e. that have only one presentation graph up to isomorphism of presentation graphs, using the notion of twist-equivalent presentation graphs in the sense of \cite{RigidityArtinCoxeter}.

\begin{corB}
	Let $A_\Gamma$ be a large-type Artin group. If $A_{\Gamma'}$ is any Artin group isomorphic to $A_\Gamma$, then $\Gamma'$ is twist-equivalent to $\Gamma$.
		In particular,  $A_\Gamma$ is rigid if and only if $\Gamma$ does not contain a separating edge with odd label. 
\end{corB}

Given the implication for the Isomorphism Problem, we feel that the problem of determining which classes of Artin groups are invariant under isomorphism should be more actively studied. At the moment, very few classes of Artin groups are known to be invariant under isomorphism, beside large-type Artin groups and their variants (see Corollary~C). In particular, it is unknown whether the class of spherical-type Artin groups is invariant under isomorphism. Let us however mention two such invariant classes: the class of two-dimensional Artin groups (which coincides with the class of Artin groups whose maximal free abelian subgroups have rank $2$), and the class of right-angled Artin groups. For the latter class, it follows from \cite{Baudisch} that two elements of a given right-angled Artin group either commute or generate a free subgroup, hence such groups cannot contain a large-type dihedral Artin subgroup. 
\medskip 

While this article focuses mainly on large-type Artin groups, our main result has implications for the Isomorphism Problem  for the more general class of two-dimensional Artin groups, which is still open. In \cite{V-RigidityLarge}, the second author showed that, given an integer $k \geq 3$, the existence in the presentation graph of an edge labelled $k$ is an invariant of isomorphism for two-dimensional Artin groups. Together with Theorem~A, this provides new information about isomorphism classes of two-dimensional Artin groups:

\begin{corC}
	Let $A_\Gamma$ be a two-dimensional Artin group. Then the set of labels of~$\Gamma$ is an invariant of isomorphism for the class of all Artin groups. 
	In particular, the families of extra-large-type Artin groups, XXL-type Artin groups, etc. are all invariant under isomorphism.
\end{corC}

We now explain the structure of this paper. In order to distinguish large-type Artin groups from other Artin groups, we introduce properties (formulated in terms of subgroups and centralisers of elements, and thus invariant under isomorphism) that distinguish the $\Z^2$-subgroups coming from an edge labelled 2 of the presentation graph from the other $\Z^2$-subgroups of the group. The approach is geometric in nature and relies on the action of Artin groups on their Deligne complex. These actions were previously used to classify the $\Z^2$-subgroups of two-dimensional Artin groups in \cite{MP-Abelian}, and to study the centralisers of elements in \cite{MP-Acylindrical, V-RigidityLarge}. 

After a section recalling the relevant results from the literature, we prove in Section~\ref{sec:key_prop} the key Proposition~\ref{prop:char-generators} that provides a purely algebraic characterisation of the conjugacy class of (most) standard generators of a large-type Artin group. This proposition is our key tool in proving Theorem~A in Section~\ref{sec:proof_thmA}.

\section{Background on Artin groups and associated complexes}

In this preliminary section, we recall some standard terminology about Artin groups, as well as the construction of the Deligne complex. We also introduce several subsets of the Deligne complex that are used to understand the structure of centralisers of elements.

\medskip

\paragraph{Artin groups and Deligne complexes}

We start by recalling the definition of an Artin group. A \textbf{presentation graph} is a finite simplicial graph  $\Gamma$ where every edge between vertices $a, b$ is labelled by an integer $m_{ab} \geq 2$. The \textbf{Artin group} associated to $\Gamma$ is the group $A_\Gamma$ given by the following presentation: 
$$  \langle a \in V(\Gamma ) ~|~ \underbrace{aba\cdots}_{m_{ab}~\mathrm{terms}} = \underbrace{bab\cdots}_{m_{ab}~\mathrm{terms}}  ~~ 
\mbox{ whenever } a,b \mbox{ are connected by an edge of }\Gamma \rangle.$$
The corresponding \textbf{Coxeter group} is the quotient of $A_\Gamma$ obtained by adding the relation $a^2 =1$ for every vertex $a$ of $\Gamma$.

We recall some standard terminology. An Artin group $A_\Gamma$ whose presentation graph $\Gamma$ is a single edge is called a \textbf{dihedral} Artin group. Given a presentation graph $\Gamma$, a \textbf{standard generator} is a vertex of $\Gamma$, i.e. one of the generators appearing in the above presentation of $A_\Gamma$.
Given an induced subgraph $\Gamma'$ of $\Gamma$, the associated \textbf{standard parabolic subgroup} is the subgroup of $A_\Gamma$ generated by the vertices of $\Gamma'$.  By a standard result of van der Lek \cite{vdL}, this subgroup is isomorphic to the Artin group $A_{\Gamma'}$, and we will henceforth use this notation to denote standard parabolic subgroups. A  \textbf{parabolic subgroup} of $A_\Gamma$ is a conjugate of a standard parabolic subgroup.

An Artin group $A_{\Gamma}$ is of \textbf{spherical type} if the corresponding Coxeter group $W_{\Gamma}$ is finite. It is of \textbf{large type} if all labels of $\Gamma$ are at least $3$, and it is \textbf{two-dimensional} if the presentation graph $\Gamma$ is not discrete and does not contain a standard parabolic subgroup on three standard generators that is of spherical type. 

The standard parabolic subgroups of an Artin group are used to define a simplicial complex that will play a key role in our proof of Theorem A:

\begin{defin}[{\cite{CD}}]
	Let $A_\Gamma$ be an Artin group. The \textbf{(modified) Deligne complex}~$D_\Gamma$ is the simplicial complex defined as follows:
	\begin{itemize}
		\item vertices of $D_\Gamma$ correspond to left cosets of standard parabolic subgroups of spherical type of $A_\Gamma$.
		\item given an element $g \in A_\Gamma$ and a chain of inclusions $\Gamma_0 \subsetneq \cdots \subsetneq \Gamma_k $ of induced subgraphs of $\Gamma$ such that each standard parabolic subgroup $A_{\Gamma_i}$ is of spherical type, the vertices 
		$$gA_{\Gamma_0}, \ldots, gA_{\Gamma_k}$$
		span a $k$-simplex of $D_\Gamma$. 
	\end{itemize}
The group $A_\Gamma$ acts on $D_\Gamma$ by multiplication on the left. 
	\end{defin}

In the case of two-dimensional Artin groups, we have the following crucial result:

\begin{lem}[{\cite{CD}}]
	If $A_\Gamma$ is a two-dimensional Artin group, then the Deligne complex $D_\Gamma$ is a two-dimensional simplicial complex that admits a piecewise-Euclidean CAT(0) metric.
\end{lem}

\bigskip 

\paragraph{Standard trees and centralisers of standard generators.} 

\begin{defin}[{\cite[Definition~4.1]{MP-Acylindrical}}]
	Let $A_\Gamma$ be a two-dimensional Artin group. Let $x \in A_\Gamma$ be an element conjugated to a standard generator. The fixed-point set~$T_x\coloneqq \mbox{Fix}(x)$ is a convex subtree of $D_\Gamma$, called a \textbf{standard tree}.
\end{defin}

We have the following elementary results about standard trees: 

\begin{lem}[{\cite[Corollary 2.18]{HMS-Artin}}]\label{lem:tree-intersection}
	Let $A_\Gamma$ be a two-dimensional Artin group. Let $T, T'$ be two distinct standard trees of $D_\Gamma$. Then $T, T'$ are either disjoint or intersect along a single vertex (corresponding to a coset of a standard dihedral parabolic subgroup).
\end{lem}

As shown in \cite{MP-Acylindrical},  the centraliser of a standard generator acts cocompactly on the standard tree of that generator, and that action can be used to give a complete description of the centraliser. This was done in \cite[Remark 4.6]{MP-Acylindrical}, and we also refer the reader to \cite[Figure 1]{CMV-Systolic} for more details and some concrete computations in the large-type case. In this article, we will not need the full description of these centralisers, but only the following result that follows directly from the explicit description given in \cite{MP-Acylindrical} and \cite{CMV-Systolic}. 

\begin{lem}[{\cite[Remark 4.6]{MP-Acylindrical}}]\label{lem:centraliser-large}
	Let $A_\Gamma$ be a two-dimensional Artin group and let $x$ be a vertex of $\Gamma$. Let $\overline{\Gamma}$ be the graph obtained from $\Gamma$ by cutting it along the middle of even-labelled edges. (In particular, each even-labelled edge of $\Gamma$ yields two distinct edges of $\overline{\Gamma}$.) We denote by $\overline{\Gamma}_x$ the connected component of $x$ in $\overline{\Gamma}$.
	
	The centraliser $C(x)$ splits as a direct product of the form~$\langle x \rangle \times F$, where~$F$ is a finitely generated free subgroup of $A_\Gamma$ of rank $\mbox{rk}(F)= |E( \overline{\Gamma}_x)| + \mbox{rk} \big(\pi_1(\overline{\Gamma}_x)\big)$.
\end{lem}

Since the structure of centralisers will play a key role in this article, we introduce the following terminology:

\begin{notation}
	We say that a group is of the form $\Z\times F_{\geq 2}$ if it is isomorphic to a direct product of the form $\Z\times F$, where $F$ is a non-abelian free group.
\end{notation}

As a direct consequence of Lemma~\ref{lem:centraliser-large}, we get:

\begin{lem}\label{lem:centraliser-valence2}
	Let $A_\Gamma$ be a two-dimensional Artin group. If $x$ is a standard generator of $A_\Gamma$ that corresponds to a vertex of $\Gamma$ of valence at least $2$, then the centraliser~$C(x)$ is of the form $\Z \times F_{\geq 2}$.
\end{lem}

\paragraph{Transverse trees and centralisers of hyperbolic elements} If $G$ is a group acting by isometries on a CAT(0) space $X$, the \textbf{minset} of an element $g \in G$ is the subspace $Min(g)$ that consists of all the points of $X$ for which the displacement induced by $g$ is minimal. When  $g \in G$ acts hyperbolically on $X$, then $Min(g)$ decomposes as a direct product $\mathcal{T} \times \mathbb{R}$, where $\mathbb{R}$ is an infinite line. The centraliser $C(g)$ acts on $X$ by preserving $Min(g)$ and $\mathcal{T}$. An \textbf{axis} of a hyperbolic element $g$ is a geodesic line of $X$ on which $g$ acts as a (non-trivial) translation. We refer the reader to \cite[Chapter II.6]{BH} for more details.

\begin{defin} [{\cite[Lemma 3.4]{V-RigidityLarge}}]
Let $A_{\Gamma}$ be a two-dimensional Artin group with Deligne complex $D_{\Gamma}$, and let $g \in A_{\Gamma}$ be an element acting hyperbolically. Then the minset of $g$ decomposes as a direct product $\mathcal{T} \times \mathbb{R}$ where $\mathcal{T}$ is a (real-)tree, called the \textbf{tranverse-tree} associated with $g$.
\end{defin}

\begin{lem} [{\cite[Lemma 3.13, Proposition 3.26, Lemma 4.3, Lemma 4.13]{V-RigidityLarge}}] \label{LemmaMinset}
Let $A_{\Gamma}$ be a large-type Artin group, and let $g \in A_{\Gamma}$ be a hyperbolic element. Then either:
\begin{itemize}
    \item the transverse tree $\mathcal{T}$ is bounded, or
    \item the transverse tree $\mathcal{T}$ is a line and $Min(g)$ is a subcomplex isometric to a tiling of the Euclidean plane by equilateral triangles, or
    \item  the element $g$ admits an axis that is contained in a standard tree of $D_\Gamma$, and $C(g)$ is isomorphic to the dihedral Artin group 
$\langle x, y ~|~ xyxy=yxyx \rangle.$
\end{itemize}
\end{lem}

\section{A partial characterisation of standard generators in large-type Artin groups}\label{sec:key_prop}

The goal of this section is to prove the following result:

\begin{prop}\label{prop:char-generators}
	Let $A_\Gamma$ be a large-type Artin group. Then an  element $x \in A_\Gamma$ has a centraliser of the form $\Z\times F_{\geq 2}$ if and only if it is conjugated to a non-trivial power of a standard generator of~$A_\Gamma$ that is neither an isolated vertex  nor a leaf of $\Gamma$  contained in an even-labelled edge. 
 
 (Moreover, we have $C(x) \cong \Z$ if $x$ is an isolated vertex of $\Gamma$, and $C(x) \cong \Z^2$ if~$x$ is a leaf of $\Gamma$  contained in an even-labelled edge.)
\end{prop}

\begin{rem} Proposition \ref{prop:char-generators} cannot be extended to two-dimensional Artin groups in general. Indeed, consider the Artin group $A_{\Gamma}$ whose presentation graph is a triangle with coefficients $(m_{ab}, m_{ac}, m_{bc}) = (2, 4, 4)$. It was proved in {\cite[Proposition 3.26, Lemma 4.1]{V-RigidityLarge}} that the element $cbca$ acts hyperbolically on the Deligne complex $D_{\Gamma}$, yet its centraliser is isomorphic to the direct product $\mathbb{Z} \times F_2$.
\end{rem}

The proof of Proposition~\ref{prop:char-generators} will be split into several lemmas that describe the centralisers of elements of $A_\Gamma$. 
The following is a direct consequence of Lemma~\ref{lem:centraliser-large}:

\begin{lem}\label{rem:char_converse} Let $A_{\Gamma}$ be a large-type Artin group, and let $g$ be a standard generator of $\Gamma$. If $g$ corresponds to an isolated vertex of $\Gamma$, then $C(g) \cong \Z$. If $g$ corresponds to a leaf of $\Gamma$  contained in an even-labelled edge, then $C(g) \cong \Z^2$. \qed
	\end{lem}

\begin{lem} \label{LemmaCentraliserHyperbolicElements}
Let $A_{\Gamma}$ be a large-type Artin group, and let $g \in A_{\Gamma}$ be an element acting hyperbolically on $D_{\Gamma}$. Then $C(g)$ is not of the form $\mathbb{Z} \times F_{\geq 2}$.
\end{lem}

\begin{proof}
By Lemma \ref{LemmaMinset} there are three possibilities. If the transverse-tree $\mathcal{T}$ is bounded, then it has a (unique) centre that is necessarily fixed by $C(g)$. In particular, $C(g)$ stabilises an axis of $g$. We know from {\cite[Lemma 3.23]{V-RigidityLarge}} that the stabiliser of an axis of a hyperbolic element is either isomorphic to $\mathbb{Z}$ or to $\mathbb{Z}^2$, depending on whether that axis is contained in a standard tree. Thus, $C(g)$ is not of the form $\mathbb{Z} \times F_{\geq 2}$.

If $Min(g)$ is an Euclidean plane tiled by equilateral triangles, then $C(g)$ acts on this plane simplicially. This action is faithful since triangle stabilisers are trivial, hence $C(g)$ is contained in the symmetry group of this tiling, which is virtually $\mathbb{Z}^2$. In particular, $C(g)$ is not of the form $\mathbb{Z} \times F_{\geq 2}$.

Finally, if $C(g)$ is isomorphic to the dihedral Artin group $\langle x, y \ | \ xyxy = yxyx \rangle$, then it is not of the form $\mathbb{Z} \times F_{\geq 2}$ since these groups have non-isomorphic abelianisations.
\end{proof}

\begin{lem} \label{LemmaCentraliserType2Elements}
Let $A_{\Gamma}$ be a large-type Artin group, and let $g \in A_{\Gamma}$ be an element acting elliptically on $D_{\Gamma}$, but that is not conjugated to a non-trivial power of a standard generator of $A_{\Gamma}$. Then $C(g)$ is not of the form~$\Z\times F_{\geq 2}$.
\end{lem}

\begin{proof}
Up to conjugation, $g$ belongs to a dihedral Artin subgroup $A_{ab}$ for some standard generators $a, b \in V(\Gamma)$ satisfying $m_{ab} < \infty$. Since $g$ is not conjugated to a non-trivial power of a standard generator, the fixed-point set $Fix(g)$ is the single vertex $v_{ab}$. The centraliser $C(g)$ preserves that vertex, hence $C(g)$ is contained in $Stab(v_{ab}) = A_{ab}$.

 If $g$ belongs to the centre of $A_{ab}$, then $C(g)$ is precisely $A_{ab}$, which is not of the form $\mathbb{Z} \times F_{\geq 2}$ as these groups have different abelianisations. Let us now assume that $g \notin Z(A_{ab})$. Recall that the centre of a large-type dihedral Artin group is infinite cyclic (\cite{BS-centre}), and consider the central quotient
$ A_{ab}/Z(A_{ab}).$
It is well-known (see for instance \cite[Section 2]{CHR-equations}) that this quotient splits as a free product of two cyclic groups. We consider the Bass-Serre tree $T$ corresponding to that splitting. The element $g$ defines a non-trivial element $\bar{g}$ of $ A_{ab}/Z(A_{ab})$. We have an exact sequence of the form
$$1 \rightarrow Z(A_{ab}) \rightarrow C(g) \rightarrow C(\bar{g}),$$
so we can study the centraliser of $\bar g$ in $ A_{ab}/Z(A_{ab})$.  There are two cases to consider:
\smallskip

\noindent \underline{Case 1}: If $\bar{g}$ acts elliptically on $T$, then $\bar{g}$ has a non-trivial fixed set $Fix(\bar{g})$. Since  stabilisers of edges of $T$ are trivial, $Fix(\bar g)$ must be a single vertex. Note that $C(\bar{g})$ must be contained in the stabiliser of this vertex, hence it is cyclic.
\smallskip

\noindent \underline{Case 2}: If $\bar{g}$ acts hyperbolically on $T$, then $\bar{g}$ has a unique axis in $T$, and this axis is stabilised by $C(\bar{g})$. The stabilisers of edges  being trivial, $C(\bar{g})$ acts faithfully on that simplicial line, hence $C(\bar{g})$ is contained in the symmetry group of a simplicial line, which is virtually $\mathbb{Z}$. 
\medskip

By the above, we know that $C(\bar{g})$ is virtually cyclic. In particular, it follows that~$C(g)$ is an extension of a virtually cyclic group (a suitable subgroup of $C(\bar{g})$) by the infinite cyclic group $Z(A_{ab})$, hence it is virtually abelian. Thus, $C(g)$ is not of the form $\Z\times F_{\geq 2}$.
\end{proof}

\begin{proof}[Proof of Proposition~\ref{prop:char-generators}]
By Lemmas \ref{LemmaCentraliserHyperbolicElements} and \ref{LemmaCentraliserType2Elements}, every element that is not conjugated to a non-trivial power of a standard generator has a centraliser that is not of the form $\Z\times F_{\geq 2}$.
\end{proof}

\begin{rem}
While this will not be used in the rest of the article, we mention the following complete classification of the centralisers of elements of $A_\Gamma$, which can be derived from the previous proofs:
\\(1) If $g$ acts elliptically on $D_{\Gamma}$ and $Fix(g)$ is the standard tree $T \coloneqq Fix(a)$, then:
\begin{itemize}
    \item If $a$ is an isolated vertex, then $C(g) \cong \mathbb{Z}$.
    \item If $a$ is the tip of an even-labelled leaf, then $C(g) \cong \mathbb{Z}^2$.
    \item Otherwise, $C(g) \cong \mathbb{Z} \times F$ for some finitely generated non-abelian free group~$F$.
\end{itemize}
(2) If $g$ acts elliptically on $D_{\Gamma}$ and $Fix(g)$ is a single vertex $v$ with stabiliser $A_{ab}$, then:
\begin{itemize}
    \item If $g \in Z(A_{ab})$, then $C(g)$ is a large-type dihedral Artin group. (e.g. $ababab$ in the $(3, 3, 3)$ Artin group.)
    \item If $g \notin Z(A_{ab})$ but $g^n \in Z(A_{ab})$ for some $n \neq 0$, then $C(g) \cong \mathbb{Z}$. (e.g. $ab$ in the $(3, 3, 3)$ Artin group.)
    \item If $g^n \notin Z(A_{ab})$ for any $n \neq 0$, then $C(g) \cong \mathbb{Z}^2$. (e.g. $ab^{-1}$ in the $(3, 3, 3)$ Artin group.)
\end{itemize}
(3) If $g$ acts loxodromically on $D_{\Gamma}$, then:
\begin{itemize}
    \item If $g$ has a bounded tranverse-tree, and some axis of $g$ is contained in a standard tree of $D_\Gamma$, 
    then $C(g) \cong \mathbb{Z}^2$. (e.g. $b^n abcabc$ for $n \neq 0$, in the $(3, 3, 3)$ Artin group.)
    \item If $g$ has a bounded tranverse-tree, and no axis of $g$ is contained in a standard tree of $D_\Gamma$, 
    then $C(g) \cong \mathbb{Z}$.
    (e.g. $(ab^{-1})^n (cb^{-1})^n$ for $n$ large, in the $(3, 3, 3)$ Artin group.)
    \item If $g$ has an unbounded tranverse-tree and has axes that are contained in standard trees, $C(g) \cong \langle x, y \ | \ xyxy = yxyx \rangle$. (e.g. $abcabc$ in the $(3, 3, 3)$ Artin group.)
    \item If $g$ has an unbounded tranverse-tree but has no axis that is contained in a standard tree, then $C(g) \cong \mathbb{Z}^2$. (e.g. $babc$ in the $(3, 3, 3)$ Artin group.)
\end{itemize}
\end{rem}

\section{A characterisation of large-type Artin groups}\label{sec:proof_thmA}

In this section, we finally prove Theorem A by introducing  properties that tell apart large-type Artin groups from Artin groups associated to presentation graphs that contain an edge labelled $2$. Since the differentiating properties will depend on the type of edge labelled $2$ we are considering, we introduce the following definition:

\begin{defin}
	Let $\Gamma$ be a presentation graph. We say that an edge $e$ of $\Gamma$ is \textbf{isolated} if both vertices of $e$ have valence $1$, is an \textbf{inner} edge if both vertices of $e$ have valence at least $2$, and is an \textbf{outer} edge if exactly one vertex of $e$ has valence~$1$.  
\end{defin}

\paragraph{The case of isolated edges with label 2.} We first introduce the differentiating property for graphs that contain an isolated edge labelled $2$.

\begin{defin}\label{def:isolated}
	Let $G$ be a group, and let $H$ be a subgroup of $G$ isomorphic to~$\Z^2$. We say that $H$ is an \textbf{isolated} $\Z^2$ if for every subgroup $H'$ of $G$ isomorphic to~$\Z^2$, we either have $H' \subseteq H$ or $H \cap H' = \{1\}$. 
	\end{defin}

\begin{lem}\label{lem:2d-isolated}
	Let $A_\Gamma$ be an Artin group, and assume that $\Gamma$ contains an isolated edge labelled $2$. Then $A_\Gamma$ contains an isolated $\Z^2$.
\end{lem}

\begin{proof}
Let $\Gamma'$ be an isolated edge labelled $2$, and let $\Gamma'' \coloneqq \Gamma - \Gamma'$. By a result of van der Lek \cite{vdL}, the parabolic subgroup $A_{\Gamma'}$ is isomorphic to $\Z^2$. Let us show that this subgroup is an isolated $\Z^2$. 

Since $\Gamma'$ is a connected component of $\Gamma$, we have a free splitting $A_\Gamma = A_{\Gamma'} * A_{\Gamma''}$. Consider the acylindrical action (without inversion) of $A_\Gamma$ on the Bass--Serre tree $T$ corresponding to that splitting, and let $v$ be the (unique) vertex of $T$ whose stabiliser is exactly $A_{\Gamma'}$. Let $H$ be a subgroup isomorphic to $\Z^2$ that intersects $A_{\Gamma'}$ non-trivially. Since a non virtually cyclic group acting acylindrically on a tree either contains $F_2$ or fixes a point, and $H$ is isomorphic to $\Z^2$ by assumption,  it follows that $H$ is contained in a vertex stabiliser. Since edge stabilisers are trivial, there exists a vertex $w$ of $T$ such that the fixed-point set of every non-trivial element of $H$ is exactly $w$. Since $H$ intersects $A_{\Gamma'} = \mbox{Stab}(v)$ non-trivially, it follows that~$v=w$, hence $H \subset A_{\Gamma'}$.
\end{proof}

\begin{lem}\label{lem:large-isolated}
	Let $A_\Gamma$ be a large-type Artin group. Then $A_\Gamma$ does not contain an isolated $\Z^2$-subgroup.
\end{lem}

\begin{proof}
	Let $H$ be a subgroup of $A_\Gamma$ isomorphic to $\Z^2$. We start by mentioning the following straightforward fact, which is left to the reader:
	
	\medskip
	
	\paragraph{Fact.} If $H$ is virtually contained in a subgroup of $A_\Gamma$ of the form $\Z \times F_{\geq 2}$, then $H$ is not an isolated~$\Z^2$-subgroup.

	\medskip
	
	The $\Z^2$-subgroups of a large-type Artin group were classified in \cite[Theorem~B]{MP-Abelian}. In particular, up to conjugation, there are three cases to consider: 

 \smallskip
	
	\underline{Case 1}: Suppose that $H$ is contained in a (large-type) dihedral parabolic subgroup generated by two adjacent vertices $a, b$ of $\Gamma$. Since $m_{ab} \geq 3$, it is a standard result that such dihedral Artin groups are virtually of the form $\Z \times F_{\geq 2}$ (see for instance \cite[Lemma~2.3]{HMS-Artin}). In particular, $H$ is virtually contained in a subgroup of $A_\Gamma$ of the form $\Z \times F_{\geq 2}$, hence it is not an isolated $\Z^2$-subgroup by the above Fact.

 \smallskip
	
	\underline{Case 2}: Suppose that $H$ is contained in the centraliser $C(a)$ of a standard generator $a \in V(\Gamma)$. Note that $a$ cannot be an isolated vertex since such elements have an infinite cyclic centraliser by Remark~\ref{rem:char_converse}. If $a$ is not a valence $1$ vertex of $\Gamma$ contained in an even-labelled edge, then the centraliser $C(a)$ is of the form $\Z\times F_{\geq 2}$ by Lemma~\ref{lem:centraliser-valence2}, and it follows from the above Fact that $H$ is not an isolated $\Z^2$. If  $a$ is a valence $1$ vertex of $\Gamma$ adjacent to a vertex $b \in V(\Gamma)$ with $m_{ab}$ even, then it follows from  \cite[Lemma~4.3]{MP-Acylindrical} that the standard tree $T_a$ is bounded, with circumcentre the vertex $v_{ab}$ corresponding to the coset $A_{ab}$ of $A_\Gamma$ (see also \cite[Remark~2.17]{V-RigidityLarge} for an explicit description of $T_a$). Since $C(a)$ stabilises the standard tree $T_a$, it also  fixes the vertex $v_{ab}$, hence $H$ is contained in the large-type dihedral parabolic $A_{ab}$, and we are back to Case 1.

 \smallskip
	
	\underline{Case 3}: There exist vertices $a, b, c$ of $\Gamma$ with $m_{ab} = m_{bc} = m_{ac} = 3$ and such that $H \subset C(abcabc)$. In that case, it follows from \cite[Lemma~4.3]{V-RigidityLarge} that $C(abcabc)$ is isomorphic to the dihedral Artin group $\langle x, y ~|~ xyxy=yxyx \rangle$. Thus, $H$ is contained in a (large-type) dihedral  Artin group, which is virtually isomorphic to $\Z \times F_{\geq 2}$ by the above. It now follows from the above Fact that $H$ is not an isolated $\Z^2$-subgroup.
\end{proof}

Since having an isolated $\Z^2$-subgroup is invariant under isomorphism, we get:

\begin{cor}\label{cor:isolated}
	If $A_\Gamma$ is an Artin group such that $\Gamma$ contains an isolated edge labelled $2$, then $A_\Gamma$  is not isomorphic to a large-type Artin group. \qed.
\end{cor}

\paragraph{The case of outer edges with label $2$.} We  introduce the differentiating property for presentation graphs that contain an outer edge labelled $2$.

\begin{defin}
	Let $A_\Gamma$ be an Artin group. We say that $A_\Gamma$ satisfies property~$(P_1)$ if for every element $x\in A_\Gamma$ whose centraliser $C(x)$ is of the form $\Z \times F_{\geq 2}$,  the centraliser $C(x)$ is contained in the normal closure of the set $\mathcal{S}_{\mathrm{dih}}$ of subgroups of $A_\Gamma$ isomorphic to non-abelian dihedral Artin groups. 
\end{defin}

We emphasise that in the previous definition, the subgroups of $A_\Gamma$ belonging to~$\mathcal{S}_{\mathrm{dih}}$ are not required to be parabolic subgroups.

\begin{lem}\label{lem:large-outer}
	Let $A_\Gamma$ be a large-type Artin group. Then $A_\Gamma$  satisfies property~$(P_1)$.
\end{lem}

\begin{proof}
	An element $x\in A_\Gamma$ whose centraliser is of the form $\Z \times F_{\geq 2}$ is conjugated to a non-trivial power of a standard generator of $A_\Gamma$ by Proposition~\ref{prop:char-generators}, and it follows from Remark~\ref{rem:char_converse} that this standard generator is not an isolated vertex of $\Gamma$. Let~$\Gamma_1, \ldots, \Gamma_k$ be the connected components of $\Gamma$, with $x$ contained in $\Gamma_1$. Since $A_\Gamma$ splits as the free product of the $A_{\Gamma_i}$, it follows that $C(x)$ is contained in $A_{\Gamma_1}$. Since $A_{\Gamma_1}$ is generated by the large-type dihedral parabolic subgroups corresponding to the edges of $\Gamma_1$, the result follows.
\end{proof}

\begin{lem}\label{lem:2d-outer}
	Let $A_\Gamma$ be a two-dimensional Artin group, and assume that $\Gamma$ contains an outer edge labelled $2$. Then $A_\Gamma$ does not satisfy property $(P_1)$.
\end{lem}

\begin{proof}
	Let $e$ be an outer edge of $\Gamma$ labelled $2$ with vertices $x, y$, where $y$ has valence $1$ and $x$ has valence at least $2$. By Lemma~\ref{lem:centraliser-valence2}, the centraliser $C(x)$ is of the form $\Z\times F_{\geq 2}$.
 We will show that the element $y \in C(x)$ is not contained in the normal closure of $\mathcal{S}_{\mathrm{dih}}$, which will show that $A_\Gamma$ does not satisfy property $(P_1)$.
	
	Let $\Gamma'$ be the subgraph of $\Gamma$ induced by $\Gamma - \{y\}$. Consider a  subgroup $H$ of $A_\Gamma$ isomorphic to a non-abelian dihedral Artin group. It follows from \cite[Theorem~D]{V-RigidityLarge} that $H$ is  conjugated either to a subgroup of a standard dihedral parabolic subgroup on two generators $a, b$, with necessarily $m_{ab}\geq 3$, or  to a subgroup of a standard parabolic subgroup on three generators $a, b, c$ with $m_{ab}=m_{bc}=m_{ac}=3$. Either way, we get that $H$ belongs to the normal closure $\langle \langle A_{\Gamma'} \rangle \rangle$ of $A_{\Gamma'}$. Thus,
 it follows that $\langle \langle \mathcal{S}_{\mathrm{dih}} \rangle \rangle\subset \langle \langle A_{\Gamma'} \rangle \rangle$.
	
	Now, consider the homomorphism $f: A_\Gamma \rightarrow \langle y \rangle$ that sends $y$ to itself and every other standard generator  of $A_\Gamma$ to the identity element. This homomorphism is well-defined since $e$ is an outer edge with even label. Since every vertex of $\Gamma'$ belongs to $\ker(f)$, we get that $A_{\Gamma'}$, and thus $\langle \langle A_{\Gamma'} \rangle \rangle$ and $\langle \langle \mathcal{S}_{\mathrm{dih}} \rangle \rangle$, belong to the kernel of $f$. Since $y \notin \ker(f)$ by construction, it follows that the element $y \in C(x)$ is not contained in $\langle \langle \mathcal{S}_{\mathrm{dih}} \rangle \rangle$, which finishes the proof. 
\end{proof}

Since property $(P_1)$ is invariant under isomorphism, we get:

\begin{cor}\label{cor:outer}
	If $A_\Gamma$ is an Artin group such that $\Gamma$ contains an outer edge labelled~$2$, then $A_\Gamma$  is not isomorphic to a large-type Artin group. \qed.
\end{cor}

\paragraph{The case of inner edges with label $2$.} We  introduce the differentiating property for presentation graphs that contain an inner edge labelled $2$.

\begin{defin}
	Let $A_\Gamma$ be a two-dimensional Artin group. We say that $A_\Gamma$ satisfies property $(P_2)$ if there exist two elements $x, y \in A_\Gamma$ such that $x, y$ generate a $\Z^2$-subgroup, and the centralisers $C(x)$ and $C(y)$ both are of the form $\Z \times F_{\geq 2}$.
\end{defin}

\begin{lem}\label{lem:2d-inner}
	Let $A_\Gamma$ be a two-dimensional Artin group, and assume that $\Gamma$ contains an inner edge labelled $2$. Then $A_\Gamma$ satisfies property $(P_2)$.
\end{lem}

\begin{proof}
	Let $x, y$ denote the vertices of an inner edge of $\Gamma$ of label $2$. Since that edge has label $2$, a standard result of van der Lek \cite{vdL} implies that the parabolic subgroup $\langle x, y\rangle$ is isomorphic to $\Z^2$. Since both $x$ and $y$ have valence at least $2$, it follows from Lemma~\ref{lem:centraliser-valence2} that the centralisers $C(x)$ and $C(y)$ are both of the form~$\Z \times F_{\geq 2}$.
\end{proof}

\begin{lem}\label{lem:large-inner}
	Let $A_\Gamma$ be a large-type Artin group. Then $A_\Gamma$ does not satisfy property $(P_2)$.
\end{lem}

\begin{proof}
	Let $x, y \in A_\Gamma$ be two elements whose centralisers are of the form $\Z \times F_{\geq 2}$, and let us show that the $x, y$ do not generate a $\Z^2$-subgroup. 
	
	Since $A_\Gamma$ is of large type, it follows from Proposition~\ref{prop:char-generators} that $x$ and $y$ are conjugates of powers of standard generators. By \cite[Corollary~2.16]{HMS-Artin}, the fixed-point set of a standard generator of $A_\Gamma$ is equal to the fixed-point set of any of its non-trivial powers. Thus, there exist two standard trees $T, T'$ of the Deligne complex $D_\Gamma$ such that $\mbox{Fix}(x)= T$ and $\mbox{Fix}(y)= T'$. By Lemma~\ref{lem:tree-intersection}, there are three cases to consider, depending on whether $T, T'$ are equal, disjoint, or intersect along a single vertex. 
	
		\underline{Case 1}: If $T=T'$, then $x$ and $y$ both stabilise some common edge of $D_\Gamma$. Since edges of standard trees of $D_\Gamma$ have infinite cyclic stabilisers by construction, it follows that $\langle x, y \rangle$ is not isomorphic to $\Z^2$. 
		
		\underline{Case 2}:	If $T$ and $T'$ are disjoint, then it follows from \cite[Proposition~C]{M-TitsWise} that there exists an integer $n \geq 1$ such that $\langle x^n, y^n \rangle$ is a non-abelian free subgroup of $A_\Gamma$, and in particular $\langle x, y \rangle$ is not isomorphic to $\Z^2$. 
	
		\underline{Case 3}: If $T$ and $T'$ intersect along a single vertex corresponding to a coset of a standard parabolic subgroup on two generators $a, b$, then up to conjugation, we can assume that $x, y$ belong to a large-type dihedral Artin subgroup on two generators $a, b$, and that $x, y$ are conjugates of powers of either $a$ or $b$. It now follows from \cite[Lemma~2.8]{HMS-Artin} that $\langle x, y \rangle$ is not isomorphic to $\Z^2$. 
\end{proof}

Since property $(P_2)$ is invariant under isomorphism, we get:

\begin{cor}\label{cor:inner}
	If $A_\Gamma$ is a two-dimensional Artin group such that $\Gamma$ contains an inner edge labelled $2$, then $A_\Gamma$  is not isomorphic to a large-type Artin group \qed.
\end{cor}

We are finally ready to prove Theorem~A. 

\begin{proof}[Proof of Theorem~A]
	Let $A_\Gamma$ be a large-type Artin group, and let $A_{\Gamma'}$ be an Artin group isomorphic to $A_\Gamma$. By \cite[Theorem~E]{V-RigidityLarge}, it follows that $A_{\Gamma'}$ is two-dimensional. By Corollaries~\ref{cor:isolated}, \ref{cor:outer}, and \ref{cor:inner}, it follows that $\Gamma'$ does not contain any edge labelled~$2$. Thus, $A_{\Gamma'}$ is also of large type.
\end{proof}

\begin{rem} As an immediate corollary, we obtain the following purely algebraic  characterisation  (i.e. independent of the presentation graph) of large-type Artin groups:
		An Artin group $A_\Gamma$ is of large type if and only if it is not cyclic, it does not contain a copy of $\Z^3$, it does not have an isolated $\Z^2$, it satisfies property $(P_1)$, and it does not satisfy property $(P_2)$. 
\end{rem}

\vspace{0.5cm}

\bigskip\noindent
\textbf{Alexandre Martin},

\noindent Address: Department of Mathematics and the Maxwell Institute for the Mathematical Sciences, Heriot-Watt University, Edinburgh EH14 4AS, UK.

\noindent Email: \texttt{alexandre.martin@hw.ac.uk}

\bigskip\noindent
\textbf{Nicolas Vaskou},

\noindent Address: School of Mathematics, University of Bristol, Woodland Road, Bristol BS8 1UG, UK.

\noindent Email: \texttt{nicolas.vaskou@gmail.com}
\end{document}